\RequirePackage[l2tabu, orthodox]{nag}

\documentclass[12pt]{amsart}
\usepackage{fullpage,url,amssymb,enumerate,colonequals}
\usepackage[all]{xy} 
\usepackage{mathrsfs} 
\usepackage{comment}

\usepackage[OT2,T1]{fontenc}
\DeclareSymbolFont{cyrletters}{OT2}{wncyr}{m}{n}
\DeclareMathSymbol{\Sha}{\mathalpha}{cyrletters}{"58}

\usepackage{color}

\newcommand{\defi}[1]{\textsf{#1}} 

\newcommand{\Aff}{\mathbb{A}}

\newcommand{\E}{\mathbb{E}}
\newcommand{\F}{\mathbb{F}}

\newcommand{\PP}{\mathbb{P}}
\newcommand{\PPdual}{\check{\mathbb{P}}}

\newcommand{\Z}{\mathbb{Z}}

\newcommand{\kbar}{{\overline{k}}}

\newcommand{\pp}{\mathfrak{p}}

\newcommand{\calM}{\mathcal{M}}
\newcommand{\calN}{\mathcal{N}}

\DeclareMathOperator{\GS}{GS}
\DeclareMathOperator{\Mon}{Mon}

\DeclareMathOperator{\Aut}{Aut}

\DeclareMathOperator{\codim}{codim}

\DeclareMathOperator{\Gal}{Gal}

\DeclareMathOperator{\Irr}{Irr}

\DeclareMathOperator{\Prob}{Prob}

\DeclareMathOperator{\Spec}{Spec}

\newcommand{\bad}{{\operatorname{bad}}}

\newcommand{\good}{{\operatorname{good}}}

\newcommand{\red}{{\operatorname{red}}}

\newcommand{\injects}{\hookrightarrow}
\newcommand{\intersect}{\cap} 
\newcommand{\isom}{\simeq}

\newcommand{\Union}{\bigcup} 

\newcommand{\isomto}{\overset{\sim}{\rightarrow}}

\newtheorem{theorem}{Theorem}[section]
\newtheorem{lemma}[theorem]{Lemma}
\newtheorem{corollary}[theorem]{Corollary}

\theoremstyle{definition}

\theoremstyle{remark}
\newtheorem{remark}[theorem]{Remark}

\newtheorem{example}[theorem]{Example}

\makeatletter
\g@addto@macro\bfseries{\boldmath} 
\makeatother

\usepackage{microtype}

\usepackage[
	backref,
	pdfauthor={Bjorn Poonen, Kaloyan Slavov}, 
]{hyperref}
\usepackage[alphabetic,backrefs,lite,nobysame]{amsrefs} 

\begin{document}

\title[Bertini irreducibility theorem]{The exceptional locus in the \\ Bertini irreducibility theorem for a morphism}
\subjclass[2020]{Primary 14D05; Secondary 14A10, 14G15.}
\keywords{Bertini irreducibility theorem, Lang--Weil bound, finite field, random hyperplane slicing}

\author{Bjorn Poonen}
\thanks{This article is published as Bjorn Poonen, Kaloyan Slavov, The Exceptional Locus in the Bertini Irreducibility Theorem for a Morphism, \emph{International Mathematics Research Notices}, rnaa182, \url{https://doi.org/10.1093/imrn/rnaa182}.  B.P.\ was supported in part by National Science Foundation grant DMS-1601946 and Simons Foundation grants \#402472 (to Bjorn Poonen) and \#550033.  K.S.\ was supported by NCCR SwissMAP of the SNSF.}
\address{Department of Mathematics, Massachusetts Institute of Technology, Cambridge, MA 02139-4307, USA}
\email{poonen@math.mit.edu}
\urladdr{\url{http://math.mit.edu/~poonen/}}

\author{Kaloyan Slavov}
\address{Department of Mathematics, ETH Z\"{u}rich,
R\"{a}mistrasse 101, 8006 Z\"{u}rich, Switzerland
}
\email{kaloyan.slavov@math.ethz.ch}

\date{August 4, 2020}

\begin{abstract}
We introduce a novel approach to Bertini irreducibility theorems over an arbitrary field, based on random hyperplane slicing over a finite field.
Extending a result of Benoist, we prove that for a morphism $\phi \colon X \to \PP^n$ 
such that $X$ is geometrically irreducible and the nonempty fibers of $\phi$ all have the same dimension,
the locus of hyperplanes $H$ such that $\phi^{-1} H$ is not geometrically irreducible has dimension at most
$\codim\phi(X)+1$. 
We give an application to monodromy groups above hyperplane sections.
\end{abstract}

\maketitle

\section{Introduction}\label{S:introduction}

Most Bertini theorems state that 
a moduli space of hyperplanes
contains a dense open subset whose points correspond to 
hyperplanes with some good property,
so if the moduli space has dimension $n$, 
the locus of bad hyperplanes has dimension at most $n-1$.
In contrast, we exhibit Bertini theorems in which the bad locus is often much smaller.

\subsection{Bertini irreducibility theorems}

We work over an arbitrary ground field $k$.
By variety, we mean a separated scheme of finite type over $k$; subvarieties need only be locally closed.
Given $\PP^n$, let $\PPdual^n$ be the dual projective space, 
so $H \in \PPdual^n$ means that $H$ is a hyperplane in $\PP^n$ (over the residue field of the corresponding point).
The following is part of a theorem of Olivier Benoist:

\begin{theorem}[cf.~\cite{Benoist2011}*{Th\'eor\`eme~1.4}]
\label{T:subvariety}
Let $X$ be a geometrically irreducible subvariety of $\PP^n$ for some $n \ge 0$.
Let $\calM_{\bad} \subseteq \PPdual^n$ 
be the constructible locus 
parametrizing hyperplanes $H$ such that $H \intersect X$ 
is not geometrically irreducible.
Then $\dim \calM_{\bad} \le \codim X + 1$.
\end{theorem}

\begin{example}
For a hypersurface $X \subset \PP^n$, Theorem~\ref{T:subvariety} says $\dim \calM_{\bad} \le 2$.
\end{example}

\begin{remark}
The bound $\codim X+1$ is best possible: see \cite{Benoist2011}*{Section~3.1}.
\end{remark}

\begin{remark}
Benoist assumes that $X$ is closed and geometrically integral, but these additional hypotheses can easily be removed.
In fact, he bounds a larger set $\calM_{\bad}$ that includes also the $H$ such that $H \intersect X$ is not generically reduced.
Additionally, he proves a best possible analogue in which hyperplanes are replaced by hypersurfaces of a fixed degree $e$.
Benoist's proof uses a degeneration to a union of hyperplanes.
\end{remark}

Our key new observation is that the statistics of random hyperplane slices over a finite field
give another way to bound the bad locus in Bertini irreducibility theorems.
We were inspired by \cite{Tao2012-langweil}, 
which used random slicing to give a proof of the Lang--Weil theorem, and by \cite{Slavov2017}, which used random slicing to refine the Lang--Weil bounds for hypersurfaces.
Using this approach, we give a new proof of Theorem~\ref{T:subvariety}, and generalize it to a setting 
analogous to that in \cite{Jouanolou1983}*{Th\'eor\`eme~6.3(4)}:

\begin{theorem}
\label{T:fibers of constant dimension}
Let $X$ be a geometrically irreducible variety.
Let $\phi \colon X \to \PP^n$ be a morphism.
Let $\calM_{\bad}$ be the 
constructible locus parametrizing hyperplanes $H \subset \PP^n$
such that $\phi^{-1} H$ is not geometrically irreducible.
If the nonempty fibers of $\phi$ all have the same dimension,
then $\dim \calM_{\bad} \le \codim \phi(X) + 1$.
\end{theorem}

\begin{example}
In the setting of Theorem~\ref{T:fibers of constant dimension}, if $\phi$ is dominant,
then the conclusion states that $\dim \calM_{\bad} \le 1$.
\end{example}

Theorem~\ref{T:fibers of constant dimension} can fail if the nonempty fibers of $\phi$ have differing dimensions:

\begin{example}
If $X \to \PP^n$ is the blow-up of a linear subspace $L \subset \PP^n$ with $\codim L \ge 2$, 
then $\calM_{\bad}$ parametrizes the hyperplanes containing $L$, 
so $\dim \calM_{\bad} = \codim L - 1$, 
but $\codim \phi(X) + 1 = 1$
(so the conclusion of Theorem~\ref{T:fibers of constant dimension} fails if $\codim L \ge 3$).
\end{example}

Nevertheless, Theorem~\ref{T:fibers of constant dimension} admits the following generalization.
Fix an algebraic closure $\kbar \supset k$.

\begin{theorem}
\label{T:complicated}
Let $X$ be a geometrically irreducible variety.
Let $\phi \colon X \to \PP^n$ be a morphism.
Let $W$ be the closed subset of $x \in X$ at which the relative dimension $\dim_x \phi$ 
is greater than at the generic point.
Let $W_1,\ldots,W_r$ be the irreducible components of $W_{\kbar}$ of dimension $\dim X - 1$.
Let $\calM_{\bad}$ be the locus of hyperplanes $H$ 
such that $\phi^{-1} H$ is not geometrically irreducible.
Let $\calN$ be the locus of hyperplanes over $\kbar$ that contain $\phi(W_i)$ for some $i$;
since $\{W_1,\ldots,W_r\}$ is Galois-stable, $\calN$ is definable over $k$.
Then $\calM_{\bad}$ differs from $\calN$ 
in a constructible set of dimension at most $\codim \phi(X) + 1$.
\end{theorem}

\subsection{Monodromy}

A generically \'etale morphism $\phi \colon X \to Y$ between integral varieties
has a monodromy group $\Mon(\phi)$, defined as the Galois group of the Galois closure of 
the function field extension $k(X)/k(Y)$;
see Section~\ref{S:monodromy} for a more general definition requiring only $Y$ to be integral.
Now suppose $Y \subset \PP^n$.
For a hyperplane $H \subset \PP^n$, 
let $\phi_H$ be the restriction $\phi^{-1}(H \intersect Y) \to H \intersect Y$.
The following theorem states that for all $H \in \PPdual^n$ outside a low-dimensional locus, 
$\Mon(\phi_H) \isom \Mon(\phi)$.

\begin{theorem}
\label{T:monodromy}
Let $\phi \colon X\to Y$ be a generically \'etale morphism with $Y$ an integral subvariety of 
$\PP^n$ 
over an algebraically closed field.
 Let $\calM_\good \subset \PPdual^n$ be the locus parametrizing hyperplanes $H$ such that 
\begin{enumerate}[\upshape (i)]
\item \label{I:H intersect Y is irreducible} 
$H \intersect Y$ is irreducible;
\item \label{I:normal and finite etale} 
the generic point of $H\cap Y$ has a neighborhood $U$ in $Y$ 
such that $U$ is normal and $\phi^{-1} U \to U$ is finite \'etale; and
\item \label{I:inclusion is isomorphism} 
the inclusion $\Mon(\phi_H) \injects \Mon(\phi)$ is an isomorphism. 
\end{enumerate}
Then the locus $\calM_\bad \colonequals \PPdual^n - \calM_\good$
is a constructible set of dimension at most 
$\codim Y+1$.
\end{theorem}

\begin{remark} 
Conditions \eqref{I:H intersect Y is irreducible} and~\eqref{I:normal and finite etale}
are needed to define the inclusion in~\eqref{I:inclusion is isomorphism}: 
see Section~\ref{S:monodromy}.
\end{remark}

\begin{example}
Let $k$ be an algebraically closed field of characteristic not $2$,
let $X = \Spec k[x_1,\ldots,x_n,y]/(y^2-x_1)$,
and let $\phi \colon X \to \Aff^n_k\subset\PP^n$ be the projection to 
$\Aff^n_k = \Spec k[x_1,\ldots,x_n]$.
Then $\calM_\bad$ is the $1$-dimensional locus consisting of the 
hyperplanes $x_1=a$ for $a \in k$.
\end{example}

\subsection{Structure of the article}
After a brief notation section, 
Theorem~\ref{T:fibers of constant dimension} is proved 
in Sections \ref{S:reduction to finite fields} to~\ref{S:counting};
see especially Lemma~\ref{L:very bad hyperplanes}.
We apply it to prove Theorems \ref{T:complicated} and~\ref{T:monodromy}
in Sections \ref{S:complicated} and~\ref{S:monodromy}, respectively.
The heart of our paper is 
the random slicing in Section~\ref{S:slicing}
and its application towards irreducibility in Section~\ref{S:counting}.

\section{Notation}

The empty scheme is not irreducible.
For a noetherian scheme $X$, let $\Irr X$ be the set of irreducible components of $X$.
If $X$ is an irreducible variety, 
let $k(X)$ be the function field of the associated reduced subscheme $X_{\red}$. 
The dimension of a constructible subset $C$ of a variety $V$ (viewed as a topological subspace)
equals the maximal dimension of a subvariety of $V$ contained in $C$;
then $\codim C \colonequals \dim V - \dim C$.

Let $S$ be a scheme, and let $X$ be an $S$-scheme.
Given a morphism of schemes $T \to S$, let $X_T$ denote $X \times_S T$;
in this context, if $T=\Spec A$, we may write $A$ instead of $\Spec A$.
If $s \in S$, let $X_s$ be the fiber of $X \to S$ above $s$.
If moreover $C$ is a constructible subset of $X$, then $C_T$ denotes the inverse image of $C$
under $X_T \to X$, and $C_s$ is defined similarly.

For a finite-type morphism $\phi \colon X\to Y$ of noetherian schemes with $Y$ irreducible,
$\phi$ is \'etale over the generic point of $Y$ if and only if
$\phi$ is \'etale over some dense open $V \subset Y$;
in this case, call $\phi$ \defi{generically \'etale}.

\section{Reduction to finite fields}
\label{S:reduction to finite fields}

We begin the proof of Theorem~\ref{T:fibers of constant dimension} 
by reducing to the case of a finite field.
There exists a finitely generated $\Z$-algebra $R \subset k$ 
such that $\phi \colon X \to \PP^n_k$ is the base change of 
a separated finite-type morphism (denoted using the same letters)
$\phi \colon X \to \PP^n_R$.
In the new notation, the original morphism is $\phi_k \colon X_k \to \PP^n_k$.
By shrinking $\Spec R$ if necessary, we may assume that for each $\pp \in \Spec R$, 
the fiber $X_\pp$ is geometrically irreducible \cite{EGA-IV.III}*{9.7.7(i)}, 
the nonempty fibers of $\phi_\pp$ all have the same dimension \cite{EGA-IV.III}*{9.2.6(iv)}, 
and $\dim \phi_\pp(X_\pp)=\dim \phi_k(X_k)$. 

Let $\calM_{\bad} \subset \PPdual^n_R$ be the subset parametrizing hyperplanes $H$ such that
$\phi^{-1} H$ is not geometrically irreducible; 
since the $\phi^{-1} H$ are the fibers of a family,
$\calM_{\bad}$ is constructible  \cite{EGA-IV.III}*{9.7.7(i)}.
If we prove Theorem~\ref{T:fibers of constant dimension} for a finite ground field,
so that $\dim \, (\calM_\bad)_\pp \le \codim \phi_\pp(X_\pp) + 1$ for every closed point $\pp$, 
then $\dim \, (\calM_{\bad})_k \le \codim \phi_k(X_k) + 1$ too.
Therefore from now on we assume that $k$ is finite.

\section{Random hyperplane slicing}
\label{S:slicing}

The following lemma is purely set-theoretic; for the time being, $X$ is just a set.

\begin{lemma}
\label{var_bound}
Let $\phi\colon X\to\PP^n(\F_q)$ be a map of sets for some $n \ge 1$.
Let $s$ be an upper bound on the size of its fibers.
For a $($set-theoretic$)$ hyperplane $H\subset\PP^n(\F_q)$ chosen uniformly at random, 
define the random variable $Z \colonequals \# (\phi^{-1} H)$. 
Then its mean $\mu$ and variance $\sigma^2$ satisfy
\begin{align*}
	\mu &= \#X \, (q^{-1} + O(q^{-2}))  \\
	\sigma^2 &= O( \#(\phi(X)) \, s^2 q^{-1}  ).
\end{align*}
\end{lemma}

\begin{proof}
For any $y \in \PP^n(\F_q)$, define
\[
	p_1 \colonequals \Prob(y \in H) = \frac{q^n-1}{q^{n+1}-1} = q^{-1} + O(q^{-2}).
\]
Similarly, for any \emph{distinct} $y,z \in \PP^n(\F_q)$, define
\[
	p_2 \colonequals \Prob(y,z \in H) = \frac{q^{n-1}-1}{q^{n+1}-1} = q^{-2} + O(q^{-3}).
\]
The mean of $Z$ is
\begin{align*}
	\mu &= \E Z = \sum_{x \in X} \Prob(\phi(x) \in H) = (\#X) \, p_1 = \#X \, (q^{-1} + O(q^{-2})), \\
\intertext{and the variance is}
	\sigma^2 &= \E(Z^2) - (\E Z)^2 \\
	&= \sum_{u,v \in X} \biggl( \Prob(\phi(u),\phi(v) \in H) - \Prob(\phi(u) \in H) \; \Prob(\phi(v) \in H) \biggr) \\
	&= \sum_{\phi(u)=\phi(v)} \left( p_1 - p_1^2 \right)
		+ \sum_{\phi(u) \ne \phi(v)} \left( p_2 - p_1^2 \right) \\
	&\le \sum_{\phi(u)=\phi(v)} p_1 + \sum_{\phi(u) \ne \phi(v)} 0 
	   \quad \textup{(we have $p_2 < p_1^2$ since $(q^{n+1}-1)^2 (p_1^2 - p_2) = q^{n-1}(q-1)^2$)} \\
	&\le \sum_{y \in \phi(X)} \#(\phi^{-1} y)^2 \, p_1 \\
	&=  O ( \#(\phi(X)) \, s^2 q^{-1} ).\qedhere
\end{align*}
\end{proof}

\section{The Lang--Weil bound}
\label{S:Lang-Weil}

We will apply Lemma~\ref{var_bound} when $\phi$ comes from a morphism of varieties over $\F_q$, 
so we need bounds on the number of $\F_q$-points of a variety.
Throughout the rest of this paper, 
the implied constant in a big-$O$ depends on the geometric complexity\footnote{Say that a variety is of complexity $\le M$ if it is a union of $\le M$ open subschemes, each cut out by $\le M$ polynomials of degree $\le M$ in $\Aff^n$ for some $n \le M$, with each pair of subschemes glued along $\le M$ open sets $D(f)$ with each $f$ given by a polynomial of degree $\le M$, with each gluing isomorphism defined by polynomials of degree $\le M$.}
but not on $q$.

\begin{theorem}[\cite{Lang-Weil1954}]
\label{T:Lang-Weil}
Let $X$ be a variety over $\F_q$.
Let $r=\dim X$.
\begin{enumerate}[\upshape (a)]
    \item \label{I:crude bound}
       We have $\# X(\F_q) = O(q^r)$.
    \item \label{I:geometrically irreducible}
       If $X$ is geometrically irreducible, then $\# X(\F_q) = q^r + O(q^{r-1/2})$.
    \item \label{I:a components} 
      More generally, if $a$ is the number of irreducible components of $X$ 
      that are geometrically irreducible of dimension $r$, then $\# X(\F_q) = aq^r + O(q^{r-1/2})$.
\end{enumerate}
\end{theorem}

\begin{proof}
Parts (a) and~(b) are Lemma~1 and Theorem~1 in \cite{Lang-Weil1954}.
As is well-known, (c) follows from (a) and~(b), 
since if $Z$ is an irreducible component that is not geometrically irreducible, 
then $Z(\F_q)$ is contained in the intersection of the geometric components of $Z$, 
which is of lower dimension.
\end{proof}

\section{Counting very bad hyperplanes}
\label{S:counting}

Consider a finite field $\F_q \supset k$. 
Call a hyperplane $H\in\PPdual^n(\F_q)$ \defi{very bad} if the number of \hbox{$\F_q$-irreducible} components of $\phi^{-1}H$ that are geometrically irreducible is not $1$.
The bound on the variance in Lemma~\ref{var_bound}
will bound the number of very bad hyperplanes,
because each such hyperplane contributes noticeably to the variance.

\begin{lemma} 
\label{L:very bad hyperplanes}
Let $X$ be a geometrically irreducible variety over a finite field $\F_q$ with a morphism $\phi\colon X\to\PP^n$
whose nonempty fibers are all of the same dimension.
Then the number of very bad hyperplanes in $\PPdual^n(\F_q)$ is
$O(q^{\codim \phi(X) +1})$.
\end{lemma}

\begin{proof}
Let $Y=\overline{\phi(X)}$.
Let $r=\dim X$ and $m=\dim Y$, so the nonempty fibers of $\phi$ have dimension $r-m$.
Consider the random variable $\#(\phi^{-1}H)(\F_q)$ for $H$ chosen uniformly at random in $\PPdual^n(\F_q)$. 
Let $\mu$ and $\sigma^2$ denote its mean and variance. 
By Lemma~\ref{var_bound} and 
Theorem~\ref{T:Lang-Weil}(\ref{I:crude bound},\ref{I:geometrically irreducible}) 
applied to $X$, $Y$, and the fibers of $\phi$,
\begin{align}
\label{E:mean using Lang-Weil}
    \mu &= (q^r + O(q^{r-1/2})) (q^{-1} + O(q^{-2})) = q^{r-1} + O(q^{r-3/2})  \\
\label{E:variance using Lang-Weil}
    \sigma^2 &= O(q^m q^{2(r-m)} q^{-1}) = O(q^{2r-m-1}). 
\end{align}

If $H$ is very bad, then $\phi^{-1} H \ne X$, 
so each irreducible component of  $\phi^{-1}H$ has dimension $r-1$,
and Theorem~\ref{T:Lang-Weil}\eqref{I:a components} implies that
$\#(\phi^{-1}H)(\F_q)$ is either $O(q^{r-3/2})$ or at least $2q^{r-1}-O(q^{r-3/2})$, 
so by \eqref{E:mean using Lang-Weil},
\[
    |\#(\phi^{-1}H)(\F_q)-\mu| 
        \geq q^{r-1}-O(q^{r-3/2}) 
        \geq \frac{1}{2} q^{r-1} \quad\textup{for large $q$}.
\]
Define $t$ so that $\dfrac{1}{2} q^{r-1} = t\sigma$.  Then
\begin{align*}
\text{Prob} \left( H\ \text{is very bad}\right)
&\leq \text{Prob} \left( |\#(\phi^{-1}H)(\F_q)-\mu|\geq t\sigma \right)
\\
&\leq \frac{1}{t^2} \quad\text{(by Chebyshev's inequality)}\\
&= \frac{4\sigma^2}{q^{2r-2}} \\
&= O(q^{1-m}) \quad \textup{(by \eqref{E:variance using Lang-Weil})}.
\end{align*} 
Multiplying by the total number of hyperplanes over $\F_q$, which is $O(q^n)$, gives $O(q^{n-m+1})$.
\end{proof}

\begin{lemma}
\label{L:deviations}
Let $\psi \colon V \to B$ be a morphism of varieties over a finite field $k$.
Suppose that $B$ is irreducible, and the generic fiber of $\psi$ is not geometrically irreducible. Call a point $b\in B(\F_q)$ \defi{very bad} if the number of 
\hbox{$\F_q$-irreducible} components of $\psi^{-1}b$ that are geometrically irreducible is not $1$.
Then there exists $c>0$ such that there exist arbitrarily large finite fields $\F_q \supset k$
such that $B(\F_q)$ contains at least
$c q^{\dim B}$ very bad points.
\end{lemma}

\begin{proof}
If $B' \to B$ is a quasi-finite dominant morphism of irreducible varieties,
and the result holds for the base change $V' \to B'$ of $V \to B$,
then the result holds for $V \to B$, 
because the image under $B' \to B$ of a set of $c' q^{\dim B}$ points of $B'(\F_q)$ 
has size at least $c q^{\dim B}$ for a possibly smaller $c$.

Let $\eta$ be the generic point of $B$.
All geometric components of $\psi^{-1} \eta$ are defined over a finite extension $K'$ of $k(B)$.
By choosing $B' \to B$ as above with $k(B') = K'$,
we may reduce to the case that all irreducible components of $\psi^{-1} \eta$
are \emph{geometrically} irreducible.
By passing to a finite extension of $k$
and replacing $B$ by an irreducible component of the base extension,
we may assume also that $B$ is geometrically irreducible.

If $\psi^{-1} \eta$ is empty, 
then there is a dense open subset $U$ of $B$ above which the fibers are empty, 
and $\# U(\F_q) = q^{\dim B} + O(q^{\dim B -1/2})$ 
by Theorem~\ref{T:Lang-Weil}\eqref{I:geometrically irreducible},
so the conclusion holds with $c=1/2$.

Otherwise $\psi^{-1} \eta$ has $\ge 2$ irreducible components.
Let $W_1$ and $W_2$ be their closures in $V$.
The locus of $b \in B$ such that 
the fibers of $W_1 \to B$ and $W_2 \to B$ above $b$
are distinct geometrically irreducible components of $\psi^{-1} b$
is constructible,
so by replacing $B$ by a dense open subvariety 
we may assume that the locus is all of $B$.
Now for any $\F_q\supset k$, all $b \in B(\F_q)$ are very bad,
and their number is $q^{\dim B} + O(q^{\dim B -1/2})$ 
by Theorem~\ref{T:Lang-Weil}\eqref{I:geometrically irreducible}.
\end{proof}

\begin{proof}[Proof of Theorem~\ref{T:fibers of constant dimension}]
By Section~\ref{S:reduction to finite fields}, we may assume that the ground field is finite.
Let $B$ be an irreducible variety contained in $\calM_{\bad}$.
Let $V \to B$ be the morphism whose fiber over a point corresponding to a hyperplane $H$ is $\phi^{-1} H$.
By Lemma~\ref{L:deviations}, 
for arbitrarily large $q$ there are at least $c q^{\dim B}$ very bad hyperplanes $H \in B(\F_q)$.
On the other hand, by Lemma~\ref{L:very bad hyperplanes}  
there are at most $O(q^{\codim Y + 1})$ very bad hyperplanes.
Thus $\dim B \le \codim Y+1$.
Since this holds for all irreducible $B \subset \calM_{\bad}$,
we obtain $\dim \calM_{\bad} \le \codim Y + 1$.
\end{proof}

\section{Proof of the most general version}
\label{S:complicated}

\begin{lemma}
\label{L:containing Y}
For a constructible set $Y \subset \PP^n$, the locus of hyperplanes containing $Y$ 
is a variety of dimension at most $\codim Y - 1$.
\end{lemma}

\begin{proof}
Let $L$ be the linear span of $Y$ in $\PP^n$.
The hyperplanes containing $Y$ are those containing $L$, 
which form a projective space of dimension $\codim L - 1 \le \codim Y - 1$.
\end{proof}

\begin{proof}[Proof of Theorem~\ref{T:complicated}]
We may assume that $k$ is algebraically closed.
By Lemma~\ref{L:containing Y}, we may ignore hyperplanes $H$ containing $\phi(X)$.
Now every irreducible component of $\phi^{-1} H$ is of dimension $\dim X - 1$.
Let $X' = X - W$.
By Theorem~\ref{T:fibers of constant dimension} applied to $X' \to \PP^n$,
it suffices to consider $H$ such that $\phi^{-1} H \intersect X'$ is geometrically irreducible.
For such $H$, the following are equivalent:
\begin{itemize}
\item $H \in \calM_{\bad}$;
\item $\phi^{-1} H$ is not irreducible;
\item $\phi^{-1} H$ contains a closed subset of $W$ of dimension $\dim X-1$;
\item $\phi^{-1} H$ contains $W_i$ for some $i$;
\item $H$ contains $\phi(W_i)$ for some $i$.\qedhere
\end{itemize}
\end{proof}


\section{Application to monodromy}
\label{S:monodromy}

Let $K$ be a field.  Fix a separable closure $K_s$ of $K$, and let $G_K = \Gal(K_s/K)$.
If $f \colon X \to \Spec K$ is finite \'etale,
let $\Mon(f)$ be the image of 
$G_K \to \Aut(X(K_s))$.
More generally, if $f \colon X \to Y$ is generically \'etale with $Y$ irreducible, 
let $f_K \colon X_K \to \Spec K$ be the generic fiber,
and define the \defi{monodromy group} $\Mon(X/Y) = \Mon(f) \colonequals \Mon(f_K)$.

Let $f \colon X\to Y$ be a degree~$d$ finite \'etale morphism of schemes.
As in \cite{Vakil2006-schubert}*{Section~3.5}, define the \defi{Galois scheme} of $f$ as 
the following open and closed $Y$-subscheme of the $d$th fibered power $X \times_Y \dots \times_Y X$:
\[
	\GS(f):= \{(x_1,...,x_d)\in X \times_Y \dots \times_Y X 
	            \mid x_i\neq x_j\ \textup{for}\ i\neq j \}.
\]

If $L/K$ is a finite separable extension, with Galois closure $\widetilde{L}$, 
and $f$ is $\Spec L\to \Spec K$, 
then $\Mon(f) \isom \Gal(\widetilde{L}/K)$ 
and any connected component $Z$ of $\GS(f)$ is isomorphic to $\Spec\widetilde{L}$.

\begin{lemma}
\label{L:mon_Z_over_Y_equals_mon_X_over_Y}
Let $f\colon X\to Y$ be a finite \'etale morphism with $Y$ irreducible. 
Let $Z$ be a nonempty open and closed subscheme of $\GS(f)$. Then $\Mon(X/Y)\simeq\Mon(Z/Y)$.
\end{lemma}

\begin{proof}
Let $K=k(Y)$.
Then $Z(K_s)$ is a union of one or more $G_K$-orbits 
in the set of bijections $\{1,\ldots,d\}\to X(K_s)$.
Thus $G_K \to \Aut(X(K_s))$ and $G_K \to \Aut(Z(K_s))$ have the same kernel,
and hence canonically isomorphic images.
\end{proof}

\begin{lemma}
\label{L:Lemma_Irr_open}
Let $f \colon X \to Y$ be an open morphism of noetherian schemes.
Suppose that $Y$ is irreducible, with generic point $\eta$.
Then there is a bijection $\Irr X \to \Irr X_\eta$ sending $Z$ to $Z_\eta$.
\end{lemma}

\begin{proof}
If $Z \in \Irr X$, then the set $Z' \colonequals X - \Union_{W \in \Irr X, \; W \ne Z} W \subset Z$
is nonempty and open in $X$, 
so $f(Z')$ is nonempty and open in $Y$, 
so $\eta \in f(Z') \subset f(Z)$, 
so $Z$ meets $X_\eta$.
By \cite{EGA-I-Springer}*{0, 2.1.13}, 
$\{ Z \in \Irr X : \textup{$Z$ meets $X_\eta$} \}$ is in bijection with $\Irr X_\eta$.
\end{proof}

\begin{lemma}
\label{L:standard_summary}
Let $Z \to Y$ be a right $G$-torsor for a finite group $G$, with $Y$ irreducible and $Z$ connected.
\begin{enumerate}[\upshape (a)]
\item \label{I:decomposition group}
Let $y \in Y$.
Let $T$ be a connected component of $Z_y$.
Let $G_T \subset G$ be the decomposition group of $T$.
Then $\Mon(T/y) \isom G_T \subset G$.

\item \label{I:Z_y connected}
The injection $\Mon(T/y) \injects G$ in \eqref{I:decomposition group} is an isomorphism 
if and only if $Z_y$ is connected.

\item \label{I:decomposition group is whole group}
If $Z$ is irreducible, then $\Mon(Z/Y)\isomto G$. 

\end{enumerate}
\end{lemma}

\begin{proof}
Part~\eqref{I:decomposition group} is just 
the usual theory of the decomposition and inertia groups specialized to the Galois \'etale case: 
use \cite{SGA1}*{V.1.3 and V.2.4}; the residue field extension is Galois.
Part~\eqref{I:Z_y connected} follows since $G$ acts transitively on $Z_y$. 
Part~\eqref{I:decomposition group is whole group} follows by applying~\eqref{I:Z_y connected} 
to the generic point $\eta$ of $Y$ and noting that $Z_\eta$ is irreducible by Lemma~\ref{L:Lemma_Irr_open}. 
\end{proof}

\begin{corollary}
\label{C:the_finite_etale_case_mon_of_fibers}
Let $f\colon X\to Y$ be a finite \'etale morphism, where $Y$ is irreducible and normal. 
Let $Z$ be a connected component of $\GS(f)$.
Let $y \in Y$.
Then there is an injection $\Mon(X_y/y) \injects \Mon(X/Y)$, 
well-defined up to an inner automorphism of $\Mon(X/Y)$,
and it is an isomorphism if and only if $Z_y$ is connected. 
\end{corollary}

\begin{proof}
Since $Y$ is normal, $Z$ is irreducible.
The proof of \cite{Fu2015}*{Proposition~3.2.10} implies that $Z\to Y$ is a $G$-torsor.  
Choose a connected component $T$ of $Z_y$.
By Lemma~\ref{L:standard_summary}, we have an injection $\Mon(T/y) \injects G \isom \Mon(Z/Y)$
which is an isomorphism if and only if $Z_y$ is connected.
By Lemma~\ref{L:mon_Z_over_Y_equals_mon_X_over_Y} applied to $X_y \to \{y\}$ and $X \to Y$,
this injection identifies with $\Mon(X_y/y) \injects \Mon(X/Y)$.
\end{proof}

\begin{proof}[Proof of Theorem~\ref{T:monodromy}]
Let $V\subset Y$ be the largest normal open subscheme
above which $\phi$ is finite \'etale. 
If $H$ satisfies (i), then (ii) holds if and only if $H\cap V$ is irreducible.
By Theorem~\ref{T:subvariety} applied to 
$Y\subset\PP^n$ and to $V \subset \PP^n$,
we may discard all $H$ that fail (i) or (ii).
Replace $X \to Y$ by $\phi^{-1} V \to V$.

Let $Z$ be a connected component of $\GS(\phi)$, and let $f$ be the morphism $Z \to Y$.
For $H$ such that $H \intersect Y$ is irreducible, let $h$ denote the generic point of 
$H\cap Y$; then the following are equivalent:

\vspace*{1cm}

\hspace*{-1.5cm}
\begin{minipage}[t]{0.5\linewidth}
\begin{itemize}
\item $H \in \calM_{\good}$; 
\item $\Mon(\phi_H) \isomto \Mon(\phi)$;
\item $\Mon(\phi_h)\isomto \Mon(\phi)$;
\item $Z_h$ is irreducible (by Corollary~\ref{C:the_finite_etale_case_mon_of_fibers}); 
\item $f^{-1}H$ is irreducible (by Lemma~\ref{L:Lemma_Irr_open}).
\end{itemize}
\end{minipage}%
\hspace*{-1.3cm}
\begin{minipage}[t]{0.5\linewidth}
\vspace*{-1cm}
\[
\xymatrix{
Z_h\ar[rr]\ar[rdd] &{}& f^{-1}H\ar[rr]\ar[rdd]&{}&Z\ar[rdd]^<<<<<<f&{}\\
{}&X_h\ar[rr]\ar[d]^{\phi_h}&{}&\phi^{-1}H\ar[rr]\ar[d]^{\phi_H}&{}&X\ar[d]^\phi\\
{}&\{h\} \ar[rr]&{}&H\cap Y\ar[rr]&{}&Y
}
\]
\end{minipage}

\bigskip

\noindent 
Since $Y$ is normal, $Z$ is irreducible, 
so by Theorem~\ref{T:fibers of constant dimension} applied to $f$, 
the last condition above fails on a constructible locus of dimension at most $\codim Y + 1$.
\end{proof}

\section*{Acknowledgments} 

We thank the referee for several helpful comments.


\begin{bibdiv}
\begin{biblist}


\bib{Benoist2011}{article}{
   author={Benoist, Olivier},
   title={Le th\'{e}or\`eme de Bertini en famille},
   language={French, with English and French summaries},
   journal={Bull. Soc. Math. France},
   volume={139},
   date={2011},
   number={4},
   pages={555--569},
   issn={0037-9484},
   review={\MR{2869305}},
   doi={10.24033/bsmf.2619},
}

\bib{EGA-I-Springer}{book}{
   author={Grothendieck, A.},
   author={Dieudonn{\'e}, J. A.},
   title={El\'ements de g\'eom\'etrie alg\'ebrique. I},
   language={French},
   series={Grundlehren der Mathematischen Wissenschaften [Fundamental Principles of Mathematical Sciences]},
   volume={166},
   publisher={Springer-Verlag, Berlin},
   date={1971},
   pages={ix+466},
   isbn={3-540-05113-9},
   isbn={0-387-05113-9},
   review={\MR{3075000}},
   label={EGA~I}, 
}

\bib{EGA-IV.III}{article}{
   author={Grothendieck, A.},
   title={\'El\'ements de g\'eom\'etrie alg\'ebrique. IV. \'Etude locale des
   sch\'emas et des morphismes de sch\'emas. III},
   journal={Inst. Hautes \'Etudes Sci. Publ. Math.},
   number={28},
   date={1966},
   issn={0073-8301},
   review={\MR{0217086 (36 \#178)}},
   label={EGA~$\hbox{IV}_3$}, 
   note={Written in collaboration with J.~Dieudonn\'e}, 
}

\bib{Fu2015}{book}{
   author={Fu, Lei},
   title={Etale cohomology theory},
   series={Nankai Tracts in Mathematics},
   volume={14},
   edition={Revised edition},
   publisher={World Scientific Publishing Co. Pte. Ltd., Hackensack, NJ},
   date={2015},
   pages={x+611},
   isbn={978-981-4675-08-6},
   review={\MR{3380806}},
   doi={10.1142/9569},
}

\bib{Jouanolou1983}{book}{
   author={Jouanolou, Jean-Pierre},
   title={Th\'eor\`emes de Bertini et applications},
   language={French},
   series={Progress in Mathematics},
   volume={42},
   publisher={Birkh\"auser Boston Inc.},
   place={Boston, MA},
   date={1983},
   pages={ii+127},
   isbn={0-8176-3164-X},
   review={\MR{725671 (86b:13007)}},
}

\bib{Lang-Weil1954}{article}{
   author={Lang, Serge},
   author={Weil, Andr{\'e}},
   title={Number of points of varieties in finite fields},
   journal={Amer. J. Math.},
   volume={76},
   date={1954},
   pages={819--827},
   issn={0002-9327},
   review={\MR{0065218 (16,398d)}},
}


\bib{SGA1}{book}{
    author={Grothendieck, Alexander}, 
     title={Rev\^etements \'etales et groupe fondamental (SGA~1)},
  language={French},
    series={Documents Math\'ematiques (Paris) [Mathematical Documents
            (Paris)], 3},
      note={S\'eminaire de G\'eom\'etrie Alg\'ebrique du Bois-Marie
            1960--61. [Geometric Algebra Seminar of Bois-Marie 1960-61];
            Directed by A.~Grothendieck.
            With two papers by M.~Raynaud.
            Updated and annotated reprint of the 1971 original [Lecture
            Notes in Math., vol.~224, Springer, Berlin]},
 publisher={Soci\'et\'e Math\'ematique de France},
     place={Paris},
      date={2003},
     pages={xviii+327},
      isbn={2-85629-141-4},
    review={\MR{2017446 (2004g:14017)}},
     label={SGA~1}, 
}


\bib{Slavov2017}{article}{
   author={Slavov, Kaloyan},
   title={An application of random plane slicing to counting $\mathbb{F}_q$-points on hypersurfaces},
   journal={Finite Fields Appl.},
   volume={48},
   date={2017},
   pages={60--67},
   issn={1071-5797},
   review={\MR{3705734}},
   doi={10.1016/j.ffa.2017.07.002},
}


\bib{Tao2012-langweil}{misc}{
    author={Tao, Terence},
     title={The Lang--Weil bound},
      date={2012-08-31},
   note={Blog, \texttt{https://terrytao.wordpress.com/2012/08/31/the-lang-weil-bound/}\phantom{i}},
}


\bib{Vakil2006-schubert}{article}{
   author={Vakil, Ravi},
   title={Schubert induction},
   journal={Ann. of Math. (2)},
   volume={164},
   date={2006},
   number={2},
   pages={489--512},
   issn={0003-486X},
   review={\MR{2247966}},
   doi={10.4007/annals.2006.164.489},
} 


\end{biblist}
\end{bibdiv}
\end{document}